\documentclass[11pt]{amsart} 
\usepackage{amsthm,amsbsy,amsfonts,amssymb,amscd}
\usepackage{bbm}
\usepackage[enableskew]{youngtab}
\usepackage[mathscr]{euscript} 

\usepackage{enumerate} 
\usepackage{tikz} 
\usetikzlibrary{calc,intersections,through, backgrounds,arrows,decorations.pathmorphing,backgrounds,positioning,fit,petri} 
\usetikzlibrary{calc}

\usepackage{tikz-cd} 
\usepackage[all]{xy}
\xyoption{knot}
\xyoption{poly}

\usepackage{hyperref}

\DeclareMathAlphabet{\mathpzc}{OT1}{pzc}{m}{it}
\newcommand{\lf}{\mathpzc{l}}

\usepackage[margin=1.5in]{geometry}

\numberwithin{equation}{section}

\newtheoremstyle{notes} {} {} {} {} {\bfseries} {.} {.5em} {}

\theoremstyle{plain}

\newtheorem{prop}[subsubsection]{Proposition}

\newtheorem{lemma}[subsubsection]{Lemma}

\newtheorem{thm}[subsubsection]{Theorem}

\newtheorem{thmA}{Theorem}

\newtheorem{defA}[thmA]{Definition}

\theoremstyle{remark}

\newtheorem{rem}[subsubsection]{Remark} 

\newtheorem{ddef}[subsubsection]{Definition} 

\swapnumbers

\usepackage{etoolbox}

\makeatletter
\pretocmd{\appendix}{\addtocontents{toc}{\protect\addvspace{10\p@}}}{}{}
\makeatother

\theoremstyle{remark}

\newtheorem{ex}[subsubsection]{Example}

\newtheoremstyle{construction} {} {} {} {} {\bfseries} { } {0pt} {}

\theoremstyle{construction}

\newcommand{\tto}{\twoheadrightarrow}
\newcommand{\mZ}{\mathbb{Z}}
\newcommand{\mT}{\mathbb{T}}
\newcommand{\soc}{\mathrm{soc}}

\newcommand{\im}{\mathrm{im}}
\newcommand{\cO}{\mathcal{O}}
\newcommand{\cF}{\mathcal{F}}
\newcommand{\rad}{\mathrm{rad}}
\newcommand{\scJ}{\mathscr{J}}
\newcommand{\scA}{\mathscr{A}}
\newcommand{\scR}{\mathscr{R}}

\newcommand{\op}{\mathrm{op}}
\newcommand{\II}{\mathbf{I}}
\newcommand{\lole}{\ell\hspace{-0.23mm}\ell}

\title[Ringel duality for ADR algebras]{Ringel duality and Auslander-Dlab-Ringel algebras}

\author{Kevin Coulembier}

\newcommand{\End}{{\rm End}}

\newcommand{\mk}{\Bbbk}
\newcommand{\Hom}{{\rm Hom}}
\newcommand{\Ext}{{\rm Ext}}


\begin{document} 
\date{} 
\begin{abstract}We introduce a new class of quasi-hereditary algebras, containing in particular the Auslander-Dlab-Ringel (ADR) algebras. We show that this new class of algebras is preserved under Ringel duality, which determines in particular explicitly the Ringel dual of any ADR algebra. As a special case of our theory, it follows that, under very restrictive conditions, an ADR algebra is Ringel dual to another one. The latter provides an alternative proof for a recent result of Conde and Erdmann, and places it in a more general setting.
	\end{abstract}

\maketitle 



\section*{Introduction}

In \cite{CPS}, Cline, Parshall and Scott introduced {\em quasi-hereditary algebras} in order to study in a unified framework the homological properties of the Schur algebra of the symmetric group and the algebras describing BGG category~$\cO$. Shortly after, Dlab and Ringel demonstrated in~\cite{DRAus} that certain algebras introduced by Auslander in~\cite{Auslander} are also quasi-hereditary. These algebras, known as {\em Auslander-Dlab-Ringel} (ADR) algebras, provided the largest known class of examples of quasi-hereditary algebras, since they associate a quasi-hereditary algebra~$\scA(R)$ to any (finite dimensional unital) algebra~$R$. Furthermore, the ADR construction shows that, in some sense, quasi-hereditary algebras determine all finite dimensional algebras.

In \cite{Ringel}, the {\em Ringel dual} of a quasi-hereditary algebra was introduced, indicating that quasi-hereditary algebras naturally form pairs.
Donkin proved in~\cite{DonkinTilting} that the Schur algebra is Ringel self-dual. In \cite{Soergel}, Soergel demonstrated that also the algebras describing category~$\cO$ are Ringel self-dual and, more generally, that the class of algebras describing parabolic category~$\cO$ is closed under Ringel duality, see \cite{dualities, prinjective} for more details. Until recently, almost nothing was known about the Ringel duals of ADR algebras. 

In \cite{Conde}, Conde and Erdmann showed that~$\scA(R)$ is Ringel dual to $\scA(R^{\op})^{\op}$ if all projective and injective $R$-modules have the same Loewy length and are all rigid. 
We can summarise these conditions as requesting that the left and right regular $R$-module, ${}_RR$ and $R_R$, are rigid (which clearly implies that all indecomposable summands have identical Loewy length). Furthermore, they show that Ringel duality between $\scA(R)$ and $\scA(R^{\op})^{\op}$ generally fails without these conditions.

In the current paper we substantially generalise the ADR procedure to construct a much larger class of quasi-hereditary algebras. The input is an algebra~$R$ along with a collection of ideals~$\II$ in $R$ satisfying certain compatibility conditions, and the resulting output is a quasi-hereditary algebra~$\scA(R,\II)$. We prove that we recover the algebra~$\scA(R)$ for a special choice of~$\II$, which thus yields an alternative proof of the main theorem in~\cite{DRAus}. There exists a natural duality~$\II\mapsto\mathring{\II}$ between the sets of systems of ideals for $R$ and~$R^{\op}$. This duality `preserves' the ADR-choice if and only if ${}_RR$ and~$R_R$ are rigid. We prove that the Ringel dual of~$\scA(R,\II)$ is always given by~$\scA(R^{\op},\mathring{\II})^{\op}$. By the above, this recovers the main result of \cite{Conde} as a special case, and provides an alternative proof.

Our construction actually also naturally includes a much wider class of algebras which are {\em standardly stratified}, a relaxation of the notion of quasi-heredity introduced in \cite{CPSbook}. Even our results on Ringel duality extend to this much more general picture.

We summarise our main results in more technical detail. For the sake of simplicity, we return to the special case of quasi-hereditary algebras to do this.

\begin{defA}
For an algebra~$R$ and $d\in\mZ_{\ge 1}$, a semisimple $d$-system of ideals is a collection
$$\II \;=\;\{I_{ij}\,|\,\,1\le i, j\le d+1 \}$$
of two-sided ideals in~$R$, such that
\begin{enumerate}[(a)]
\item $I_{ij}I_{jk}\subset I_{ik}$,\hspace{2mm} for $1\le i,j,k\le d+1$;
\item $I_{ij}=R$,\hspace{2mm} for $1\le  j\le i\le d+1$;
\item $R/I_{i,i+1}$ is a semisimple algebra for $1\le i\le d$.
\end{enumerate}
\end{defA}
To this data, we can associate a natural algebra structure on the vector space 
$$\scA(R,\II):=\bigoplus_{1\le i,j\le d} I_{ij}/I_{i,d+1},\quad\mbox{from multiplication}\quad I_{ij}/I_{i,d+1}\otimes I_{jk}/I_{j,d+1}\to I_{ik}/I_{i,d+1}.$$
Furthermore, for such a system $\II$ in $R$, we have a semisimple $d$-system of ideals in~$R^{\op}$ given by
$$\mathring{\II} \;=\;\{\mathring{I}_{ij}:=I_{d+2-j,d+2-i}\,|\,\,1\le i, j\le d+1 \}.$$

\begin{thmA}\label{Main1}
The algebra~$\scA(R,\II)$ is quasi-hereditary and its Ringel dual is~$\scA(R^{\op},\mathring{\II})^{\op}$.
\end{thmA}

Let~$\scJ$ be the Jacobson radical of~$R$ and~$d$ its nilpotency index.  We define the `Jacobson system'~$\II^{\scJ}$ of~$R$ as 
$$I^{\scJ}_{ij}:=\{x\in R\,|\, x\scJ^{d+1-j} \subset \scJ^{d+1-i}\},\qquad\mbox{for $1\le i,j\le d+1$.}$$

\begin{thmA}\label{Main2} Consider an algebra~$R$.
\begin{enumerate}[(i)]
\item The Jacobson system is a semisimple system of ideals in~$R$.
\item The ADR algebra~$\scA(R)$ is isomorphic to $\scA(R,\II^{\scJ})$.
\item The system~$\mathring{\II}^{\scJ}$ is the Jacobson system of~$R^{\op}$ if and only if $R_R$ and~${}_RR$ are rigid.
\end{enumerate}
\end{thmA}

The combination of Theorems \ref{Main1} and \ref{Main2}(iii) implies that $\scA(R)$ is Ringel dual to $\scA(R^{\op})^{\op}$ when~$R_R$ and~${}_RR$ are rigid.

The remainder of the paper is organised as follows. In Section~\ref{SecPrel}, we recall some ring theoretic definitions. In Section~\ref{SecIdeals}, we investigate the new notion of systems of ideals in arbitrary rings. In Section~\ref{SecDefA}, we construct the ring~$\scA(R,\II)$ out of a system of ideals~$\II$ in a ring~$R$, and show that one specific choice of~$\II$ recovers the ADR construction. We also prove that $\scA(R,\II)$ always has an interesting stratification, recovering as a special case the main result in~\cite{DRAus}. In Section~\ref{SecTilt}, we introduce an $\scA(R,\II)$-$\scA(R^{\op},\mathring{\II})^{\op}$ bimodule~$\mT$, yielding a double centraliser property, and investigate some of its homological properties. In Section~\ref{SecAlg}, we restrict to {\em semisimple} systems of ideals and the case where $R$, and hence $\scA(R,\II)$, is actually an {\em algebra} over a field. All the above results then immediately imply that $\scA(R,\II)$ is quasi-hereditary, $\mT$ its tilting module and~$\scA(R^{\op},\mathring{\II})^{\op}$ the Ringel dual. Finally, in Section~\ref{SecEx}, we calculate the explicit form of all algebras~$\scA(R,\II)$, for~$\II$ ranging over all semisimple systems of ideals in one 2-dimensional hereditary algebra~$R$. This demonstrates that our construction is a substantial generalisation of the ADR construction.


\section{Preliminaries}\label{SecPrel}
We take the convention that `ring' means {\em unital} ring. We also fix an arbitrary field $\mk$ for the entire paper. By `algebra' we will always mean {\em finite dimensional, associative, unital} algebra over~$\mk$.

\subsection{Loewy filtrations} Fix a ring $R$.
\subsubsection{}A filtration of length $k$ of an~$R$-module~$M$
$$0= F_0M\subset F_1M\subset F_2M\subset\cdots\subset F_kM=M$$
is called semisimple if all subquotients~$F_iM/F_{i-1}M$ are semisimple $R$-modules. A {\bf Loewy filtration} of a module~$M$ is a finite semisimple filtration of minimal length. That minimal length, if it exists, is the {\bf Loewy length} of~$M$, $\lole(M)$. If there are no finite semisimple filtrations, we set~$\lole(M)=\infty$. A module of finite length is {\bf rigid} if it only has one Loewy filtration.

\subsubsection{} There are two extremal Loewy filtrations for an~$R$-module~$M$ of finite length.
The {\bf socle filtration},
$$0=\soc_0M\subset \soc_1M\subset \soc_2M\subset\cdots\subset \soc_{\lole(M)} M= M,$$
is the filtration where $\soc_kM$ is the submodule of~$M$ such that $\soc_kM/\soc_{k-1}M$ is the socle of~$M/\soc_{k-1}M$.
The {\bf radical filtration},
$$0=\rad^{\lole(M)}M\subset \cdots\subset \rad^2M\subset\rad^1M\subset \rad^0{M}=M,$$
is the filtration where $\rad^iM$ is the radical of~$\rad^{i-1}M$.

Clearly a module~$M$ of finite length is rigid if and only if $\rad^i(M)=\soc_{\lole(M)-i}(M)$, for all~$0\le i\le \lole(M)$.

\subsection{Semiprimary rings}
The Jacobson radical $\mathscr{J}:=\rad(R)$ of a ring~$R$, is the ideal of elements which annihilate all simple (left, or equivalently, right) $R$-modules, see~\cite[\S4]{Lam}.
A ring~$R$ is {\bf semiprimary} if the Jacobson radical $\mathscr{J}$ is nilpotent and~$R/\scJ$ is semisimple. Any finite dimensional algebra over $\mk$ is semiprimary, see \cite[Corollary~4.19]{Lam}.

In this subsection, we assume that $R$ is semiprimary. 
\subsubsection{}\label{socradJ} It follows easily that for any right $R$-module~$M$, we have
$$\rad M_R\;=\; M\scJ\qquad\mbox{and}\qquad\soc M_R\;=\; \{v\in M\,|\, v\scJ=0\}.$$
The corresponding observation for left modules shows that ${}_RR$ and $R_R$ have the same finite Loewy length $\lole(R)$, which we call the Loewy length of~$R$.

\subsubsection{}\label{DefRigid} We call~$R$ {\bf rigid} if the left and right regular module, ${}_RR$ and~$R_R$ are rigid. By the observations in \ref{socradJ}, we have the following lemma.
\begin{lemma}\label{LemRigid} A semiprimary ring $R$ of Loewy length $\ell=\lole(R)$ is rigid if and only if
$$\{x\in R\,|\, \scJ^{\ell-i}x= 0\}\;=\;\scJ^i \;=\;\{x\in R\,|\, x\scJ^{\ell-i}= 0\},\qquad\mbox{for all $\;1\le i\le \ell$.}$$
\end{lemma}

\subsection{Quasi-hereditary algebras} In this subsection we recall some results from \cite{CPS, Ringel}.

Consider a (finite dimensional unital associative) algebra~$A$ over $\mk$. 
We assume the isoclasses of simple left $A$-modules are labelled by the (finite) set~$\Lambda=\Lambda_A$. When we consider a partial order $\le$ on~$\Lambda_A$, we will write $(A,\le)$. We denote by~$A$-mod the category of finite dimensional left $A$-modules. The simple modules are denoted by~$\{L(\lambda)\,|\,\lambda\in\Lambda\}$ and their respective projective covers in $A$-mod by~$P(\lambda)$.

\begin{ddef}
An algebra~$(A,\le)$ is {\bf quasi-hereditary} if there exist~$\{\Delta(\lambda)\,|\, \lambda\in \Lambda\}$ in $A$-mod such that, for all $\lambda,\mu\in\Lambda$:
\begin{enumerate}[(a)]
\item $[\Delta(\lambda):L(\lambda)]=1$;
\item $[\Delta(\lambda):L(\mu)]=0$ unless~$\mu\le\lambda$;
\item there is an epimorphism $P(\lambda)\tto \Delta(\lambda)$ where the kernel has a filtration with each quotient of the form $\Delta(\nu)$, with $\nu\ge \lambda$.
\end{enumerate}
\end{ddef}
By the above definition, being quasi-hereditary is a property of the Morita equivalence class of an algebra.

We denote by~$\cF^\Delta_A$ the full subcategory of~$A$-modules which have a filtration with each quotient of the form $\Delta(\nu)$, for some $\nu\in\Lambda$. The number of times~$\Delta(\nu)$ appears in the filtration of~$M\in \cF^\Delta_A$ is denoted by~$(M:\Delta(\nu))$. This is clearly independent of the chosen filtration. By \cite[Theorem~4]{Ringel}, a module~$M\oplus N$ is in $\cF^\Delta_A$ if and only if both $M$ and $N$ are in $\cF^\Delta_A$.

\subsubsection{Ringel duality}\label{SecRingel} We can reformulate some of the results in~\cite[Sections~4 and~5]{Ringel} as follows. For a quasi-hereditary algebra~$(A,\le)$ and any~$\lambda\in\Lambda$ there exists precisely one (up to isomorphism) indecomposable module~$T(\lambda)$ such that
\begin{enumerate}[(a)]
\item $\Ext^1_A(\Delta(\mu),T(\lambda))=0$, for all $\mu\in\Lambda$;
\item there exists a monomorphism $\Delta(\lambda)\hookrightarrow T(\lambda)$ with cokernel in $\cF_A^\Delta$.
\end{enumerate}
We refer to the modules~$T(\lambda)$ as {\bf tilting modules} and set~$T:=\bigoplus_{\lambda\in\Lambda}T(\lambda)$.

The simple modules of the algebra
$$\scR(A,\le):=\End_{A}(T)^{\op}$$
 are naturally labelled by~$\Lambda$, where the projective cover of the module corresponding to $\lambda$ is given by~$\Hom_A(T,T(\lambda))$. Denote by~$\le^{-1}$ the partial order on~$\Lambda$ defined by~$\mu\le^{-1}\lambda$ if and only if $\lambda\le\mu$. By \cite[Theorem~6]{Ringel}, the algebra~$(\scR(A,\le),\le^{-1})$ is quasi-hereditary and by \cite[Theorem~7]{Ringel}, $\scR(\scR(A,\le),\le^{-1})$ is Morita equivalent to $A$. The algebra~$(\scR(A,\le),\le^{-1})$ is known as the {\bf Ringel dual} of~$(A,\le)$. Ringel duality can clearly be interpreted as a duality between Morita equivalence classes of quasi-hereditary algebras. By `the' Ringel dual algebra we will refer to any algebra~$(B,\le)$ Morita equivalent to $\scR(A,\le)$.

\subsubsection{}\label{DefChain}
Assume that, for some $d\in \mZ_{\ge 1}$, we have mutually orthogonal idempotents $\{e_i\,|\, 1\le i\le d\}$ in $A$ which sum up to $1$.
Then we
have a chain of two-sided idempotent ideals in $A$
$$0=J_d\subset J_{d-1}\subset\cdots\subset J_1\subset J_0=A,$$
with $J_j$ generated by $e_d+d_{d-1}+\ldots+e_{j+1}$.
This defines a map $\lf:\Lambda\to\{1,2,\ldots, d\}$, where for each $\lambda\in\Lambda$, we have
$$J_{\lf(\lambda)}L(\lambda)=0,\quad\mbox{but}\quad J_{\lf(\lambda)-1}L(\lambda)\not=0.$$
This generates a partial order on~$\Lambda$, by setting~$\mu<\lambda$ if $\lf(\mu)<\lf(\lambda)$.
\begin{lemma}\label{LemCPS}
Assume that, for each $1\le j\le d$, the left $A/J_{j}$-module~$J_{j-1}/J_{j}$ is projective and $e_j A/J_{j}e_j$ is a semisimple algebra, then $(A,\le)$ is quasi-hereditary with standard modules
$$\Delta(\lambda)\;=\; P(\lambda)/J_{\lf(\lambda)}P(\lambda).$$ Furthermore, $\{\lambda\in\Lambda\,|\,\lf(\lambda)=i\}$ is in bijection with the isoclasses of simple $e_i A/J_{i}e_i$-modules.
\end{lemma}
\begin{proof}
That $(A,\le)$ is quasi-hereditary is \cite[Lemma~3.4]{CPS}.

Now take $\lambda\in\Lambda$ with $i=\lf(\lambda)$.
It then follows that $e_i L(\lambda)$ is non-zero and thus a simple $e_iAe_i$-module with trivial $e_iJ_ie_i$-action. If we have an isomorphism $e_iL(\lambda)\simeq e_iL(\nu)$ some $\nu$ with $\lf(\nu)=i$, it follows that $L\simeq L'$. Hence $\{\lambda\,|\,\lf(\lambda)=i\}$ embeds into the set of isoclasses of simple $e_i(A/J_i)e_i$-modules.
Furthermore, starting from a simple $e_i (A/J_i) e_i$-module~$M$, we can induce an $A$-module 
$$N:=(A/J_i)e_i\otimes_{e_i A/J_i e_i}M$$
which satisfies $J_iN=0$ and $e_iN\simeq M$. Therefore, $M$ must have a simple constituent $L$ such that the above procedure yields the simple $e_i A/J_i e_i$-module~$M$. In conclusion, the embedding is actually a bijection.
\end{proof}

\begin{lemma}\label{LemMM}
Keep the assumptions of Lemma~\ref{LemCPS}. For any~$M\in \cF_A^{\Delta}$ and~$1\le k< d$, both $A$-modules~$J_kM$ and~$M/J_kM$ are in $\cF^\Delta_A$.
\end{lemma}
\begin{proof}
By~\cite[Lemma~3.2(d)]{CPS}, for the given data there exists a short exact sequence
$$0\to M'\to M\to M''\to 0,$$
such that $M',M''\in\cF_A^\Delta$ and furthermore, $(M':\Delta(\nu))=0$ unless~$\lf(\nu)> k$ and $(M'':\Delta(\nu))=0$ unless~$\lf(\nu)\le k$. It then follows by definition of quasi-hereditary algebras that
$$\Hom_A(P(\lambda),M'')=0,\qquad\mbox{if $\lf(\lambda)>k$} $$
and furthermore, that the $A$-module~$M'$ is generated by vectors in the image of morphisms~$P(\lambda)\to M'$, for $\lf(\lambda)>k$. Hence we find that $M'=J_kM$.
\end{proof}


\section{Systems of ideals in rings}\label{SecIdeals}
We fix an arbitrary ring~$R$.
\subsection{Definitions}

\begin{ddef}\label{DefMultSyst}
For $d\in\mZ_{\ge 1}$, a {\bf $d$-system of ideals} in~$R$ is a collection 
$$\II \;=\;\{I_{ij}\,|\,\,1\le i, j\le d+1 \}$$
of two-sided ideals in~$R$, such that for all $1\le i,j,k\le d+1$
\begin{enumerate}[(a)]
\item $I_{ij}I_{jk}\subset I_{ik}$;
\item $I_{ij}=R$,\hspace{4mm} if $i\ge j$.
\end{enumerate}
\end{ddef}


\begin{ex}\label{Ex1}
We consider $d=1$. The ideals in~$R$ are in bijection with the $1$-systems. For any ideal $I$, we have the $1$-system~$\II$ given by~$I_{11}=R=I_{21}=I_{22}$ and~$I_{12}=I$.
\end{ex}

\begin{ex}\label{Ex2}
We consider $d=2$. The $2$-systems are in bijection with triples $\{I,K,L\}$ of ideals in~$R$ satisfying
$$KL\subset I\qquad\mbox{and}\qquad K\supset I\subset L.$$
Any 2-system of ideals is determined by such a tripe $\{I,K,L\}$ as
$$I_{12}=K,\; I_{13}=I,\; I_{23}=L\qquad\mbox{and}\qquad I_{ij}=R\;\mbox{ otherwise.}$$
\end{ex}

\begin{ex}\label{exaK}
Let~$L$ be a nilpotent ideal in~$R$ with nilpotency index $d$. So we have $L^d=0$ and~$L^{d-1}\not=0$, and use the convention~$L^0=R$. Define 
$$I^{L}_{ij}:=\{x\in R\,|\, xL^{d+1-j} \subset L^{d+1-i}\},\qquad\mbox{for $1\le i,j\le d+1$.}$$
Then~$\II^L $ forms a $d$-system of ideals in~$R$.

\end{ex}

\begin{lemma}\label{LemK}
For the $d$-system of Example~\ref{exaK} and $1\le i\le d+1$, we have inclusions
$$L^{d+1-i}=I^L_{i,d+1}\subset I^L_{i-1,d}\subset\cdots\subset I^L_{1,d+2-i}=\{x\in R\,|\, xL^{i-1}=0\}.$$
\end{lemma}

\subsubsection{}We are particularly interested in the case where the nilpotent ideal $L$ is the Jacobson radical. If $R$ is semiprimary with Jacobson radical $\scJ$, we call the system of ideals~$\II^{\scJ}$ the {\bf Jacobson system}. Note that $\II^{\scJ}$ is an~$\ell$-system, for $\ell:=\lole(R_R)=\lole({}_RR)$.

\begin{ddef}\label{DefSS}
A $d$-system of ideals in~$R$ is {\bf semisimple} if the ring
$R/I_{k,k+1}$ is semisimple for all~$1\le k\le d$.
\end{ddef}

\begin{lemma}
For a semiprimary ring~$R$, the Jacobson system~$\II^{\scJ}$ is semisimple.
\end{lemma}
\begin{proof}
It follows from Lemma~\ref{LemK} for $i=d$ and $L=\scJ$ that 
$$\scJ=I^{\scJ}_{d,d+1}\subset I^{\scJ}_{d-1,d}\subset\cdots\subset I_{k,k+1}^{\scJ}\subset\cdots \subset I_{1,2}^{\scJ}.$$
In particular, $R/I_{k,k+1}^{\scJ}$ is a quotient of the semisimple ring~$R/\scJ$ and thus also semisimple.
\end{proof}



\subsection{Elementary properties}

\begin{lemma}\label{LemInc}
If $\II $ is a system of ideals in $R$, then for all $1\le i,j,k\le d+1$ we have
\begin{enumerate}[(i)]
\item $I_{ij}\subset I_{ik}$,\hspace{2mm} if $j\ge k$;
\item $I_{jk}\subset I_{ik}$,\hspace{2mm} if $j\le i$.
\end{enumerate}
\end{lemma}
\begin{proof}
Part (i) follows by conditions (b) and (a) in Definition~\ref{DefMultSyst}, since
$$I_{ij}=I_{ij}R=I_{ij}I_{jk}\subset I_{ik}.$$
Part (ii) follows similarly.
\end{proof}

\begin{lemma}\label{LemTrun}
Let~$\II $ be a  $d$-system of ideals in~$R$. Then,
$$\II'\;:=\;\{I_{ij}\,|\, 1\le i,j\le d\}$$
is a  $d-1$-system of ideals in~$R$.
\end{lemma}

\subsection{Duality on the set of~$d$-systems of ideals}

There is a clear correspondence between the $d$-systems of ideals in a ring and its opposite.

\begin{lemma}
For $d$-system~$\II $ of ideals in~$R$,
we define ideals in~$R^{\op}$
$$\mathring{I}_{ij}\;:=\; I_{d+2-j,d+2-i}.$$
Then~$\mathring{\II}$ is a $d$-system of ideals in~$R^{\op}$, and~$\mathring{\mathring{\II}}=\II $. Furthermore, $\mathring{\II}$ is semisimple if and only if $\II$ is semisimple.
\end{lemma}

\begin{prop}
A semiprimary ring~$R$ is rigid if and only if $\mathring{\II}^{\scJ}$ is the Jacobson system of~$R^{\op}$.
\end{prop}
\begin{proof}
We denote by~$\II^{\scJ}$ the Jacobson system for $R$ and by~$\widehat{\II}^{\scJ}$ the one for $R^{\op}$.

Assume first that $\widehat{\II}^{\scJ}=\mathring{\II}^{\scJ}$. In particular, we have
$$I^{\scJ}_{i,d+1}=\widehat{I}^{\scJ}_{1,d+2-i}\quad\mbox{and}\quad I^{\scJ}_{1,d+2-i}=\widehat{I}^{\scJ}_{i,d+1},\quad\mbox{for all~$1\le i\le d+1$}.$$
By Lemmata~\ref{LemK} and~\ref{LemRigid}, this means that $R$ is rigid.

Now assume that $R$ is rigid.
Lemmata~\ref{LemK} and~\ref{LemRigid} then imply equalities
$$\scJ^{d+2-i}=I^{\scJ}_{i,d+1}=I^{\scJ}_{i-1,d}=\cdots=I^{\scJ}_{1,d+2-i}\,,\qquad\mbox{for all~$1\le i\le d+1$,}$$
and the same equalities for $\widehat{\II}$. In particular, for $i\le j$, we have
$$\mathring{I}^{\scJ}_{ij}=I^{\scJ}_{d+2-j,d+2-i}=I^{\scJ}_{d+1+i-j,d+1}=\scJ^{d+2-i}=\widehat{I}^{\scJ}_{d+1+i-j,d+1}=\widehat{I}^{\scJ}_{ij}$$
and thus~$\mathring{\II}^{\scJ}=\widehat{\II}^{\scJ}$.
\end{proof}


\section{A generalisation of the ADR procedure}\label{SecDefA}

\subsection{Definition}\label{SecDefAlg}
Fix an arbitrary ring $R$ for this subsection.
\subsubsection{}\label{DefX} Consider a $d$-system~$\II $ of ideals in~$R$. By Lemma~\ref{LemInc}(i), we have $I_{i,d+1}\subset I_{ij}$, for all~$i,j$. We define abelian groups
$$X_{ij}:=I_{ij}/I_{i,d+1},\qquad\mbox{for all~$1\le i,j\le d$}.$$
Then we have group homomorphisms
$$\varphi_{ijkl}:X_{ij}\times X_{kl}\to X_{il},\; (x+ I_{i,d+1},y+I_{k,d+1})\mapsto \delta_{jk}(xy +I_{i,d+1}),$$
for $x\in I_{ij}$ and~$y\in I_{kl}$.
To argue this is well-defined we can assume $j=k$. We have $xy\in I_{il}$ by Definition~\ref{DefMultSyst}(a). Furthermore, for any~$z\in I_{j,d+1}$, we have $xz\in I_{i,d+1}$ by Definition~\ref{DefMultSyst}(a), so the right-hand side does not depend on the representative of~$y+I_{k,d+1}$. The right-hand side does not depend on the representative of~$x+I_{i,d+1}$ since $I_{i,d+1}$ is an ideal.

\begin{ddef}\label{DefAlg}
For a $d$-system~$\II $ of ideals in~$R$, the abelian group
$$\scA(R,\II ):=\bigoplus_{1\le i,j\le d}X_{ij},$$
is equipped with product $ab=\varphi_{ijkl}(a,b)$ for $a\in X_{ij}$ and~$b\in X_{kl}$.
\end{ddef}
It follows easily that this product makes~$\scA(R,\II )$ into a monoid and consequently that $\scA(R,\II )$ is a ring. Note that distributivity follows automatically since $\varphi_{ijkl}$ is a group homomorphism.

A special case of the above construction also appears, although for $d=\infty$, in~\cite{Borel}. 

\begin{ex}
For $d=1$, see Example~\ref{Ex1}, we simply have
$$\scA(R,\II)\;=\; R/I.$$
\end{ex}

It can be convenient to think about the rings~$\scA(R,\II)$ as rings consisting of~$d\times d$-matrices, with $(i,j)$-entries taking values in~$X_{ij}$.
\begin{ex}
For $d=2$, with notation as in Example~\ref{Ex2}, we have
$$\scA(R,\II)\;=\; \left(\begin{array}{cc} R/I&K/I\\R/L&R/L \end{array}\right).$$
\end{ex}

\subsection{Example: the ADR procedure}

Let~$R$ be a semiprimary ring of Loewy length $d=\lole(R)$, with Jacobson radical $\scJ$. We view $R/\scJ^i$ as a right $R$-module. The ADR ring~$\scA(R)$ of \cite{Auslander, DRAus} is defined as
$$\scA(R)\;:=\; \End_R\left(\bigoplus_{i=1}^d (R/\scJ^i)_R\right).$$

\begin{prop}\label{PropADR}
For  $R$ a semiprimary ring, we have a ring isomorphism
$$\scA(R)\;\simeq\; \scA(R,\II^{\scJ}).$$
\end{prop}
\begin{proof}
Consider the exact sequence
$$0\to \Hom_R(R/\scJ^{d+1-j},R/\scJ^{d+1-i})\to \Hom_R(R,R/\scJ^{d+1-i})\to  \Hom_R(\scJ^{d+1-j},R/\scJ^{d+1-i}). $$
It follows that we have group isomorphisms
$$\Hom_R(R/\scJ^{d+1-j},R/\scJ^{d+1-i})\,\stackrel{\sim}{\to}\,\{x\in R\,|\,x\scJ^{d+1-j}\subset \scJ^{d+1-i} \}/\scJ^{d+1-i},$$
given by $\alpha\mapsto \alpha(1+\scJ^{d+1-j})$.
By Lemma~\ref{LemK}, we have $I_{i,d+1}^{\scJ}=\scJ^{d+1-i}.$ By the definition of~$I^{\scJ}$ in Example~\ref{exaK}, we thus find
isomorphisms
$$\Hom_R(R/\scJ^{d+1-j},R/\scJ^{d+1-i})\;\stackrel{\sim}{\to}\; X_{ij}.$$
and thus a group isomorphism between~$\scA(R)$ and~$\scA(R,\II^{\scJ})$.

Since the product in $\scA(R)$ corresponds to
$$\Hom_R(R/\scJ^{d+1-j},R/\scJ^{d+1-i})\times \Hom_R(R/\scJ^{d+1-k},R/\scJ^{d+1-j})$$
$$\to\Hom_R(R/\scJ^{d+1-k},R/\scJ^{d+1-i}),\quad (\alpha,\beta)\mapsto \alpha\circ \beta,$$
it follows that the isomorphism intertwines the product on~$\scA(R)$ and the product on~$\scA(R,\II^{\scJ })$ defined in Section~\ref{SecDefAlg}.
\end{proof}

\subsection{Chain of ideals}
Consider $\scA(R,\II )$ as in Section~\ref{SecDefAlg}, for a $d$-system of ideals in~$R$.
\subsubsection{} For all~$1\le k\le d$, we define
$$e_k:=1+I_{k,d+1}\in R/I_{k,d+1}=X_{kk}\subset \scA(R,\II ). $$
Then we have
$$1=\sum_{k=1}^d e_k\qquad\mbox{and}\qquad e_ke_l=\delta_{kl}e_k,\;\;\mbox{for $1\le k,l\le d$}.$$
We set~$f_j:=\sum_{k> j}e_k$, for $0\le j<d$, and $f_d=0$. In particular, we have $f_0=1$. We define the corresponding idempotent ideals~$J_j:=\scA(R,\II)f_j\scA(R,\II)$.

\begin{thm}\label{ThmChain}Set~$A:=\scA(R,\II)$ for $\II$ a $d$-system in~$R$. 
The chain of idempotent ideals
$$0=J_d\subset J_{d-1}\subset J_{d-2}\subset\cdots\subset J_1\subset J_0=A,$$
is such that the left $A/J_{i}$-module~$J_{i-1}/J_{i}$ is projective and we have a ring isomorphism
$$e_i A/J_{i} e_i\;\simeq\; R/I_{i,i+1},\qquad\mbox{for all $1\le i\le d$.}$$
\end{thm}

By Proposition~\ref{PropADR}, the restriction of Theorem~\ref{ThmChain} to semisimple systems of ideals, generalises (and provides an alternative proof for) the main theorem in~\cite{DRAus}.
We start by proving the following proposition.

\begin{prop}\label{Prop4Iter}
Set~$A:=\scA(R,\II)$ for $\II$ a $d$-system in~$R$. 
\begin{enumerate}[(i)]
\item We have an isomorphism of left $A$-modules~$J_{d-1}=Ae_dA\;\simeq\; (Ae_d)^{\oplus d}.$
\item We have a ring isomorphism $e_d Ae_d\;\simeq\; R/I_{d,d+1}$
\item We have a ring isomorphism
$$A/(Ae_dA)\;\simeq\; \scA(R,\II'),$$
with $\II'$ the $d-1$-system of Lemma~\ref{LemTrun}.
\end{enumerate}
\end{prop}
\begin{proof}
For arbitrary~$1\le i,j\le d$, we have group isomorphisms
$$e_iAe_dAe_j\;=\;\im\varphi_{iddj}\;\simeq\;(I_{id}I_{dj})/I_{i,d+1}\;=\;I_{id}/I_{i,d+1}\;\simeq\; e_iAe_d,$$
which follow by definition of $\scA(R,\II)$ and the fact that $I_{dj}=R$ by Definition~\ref{DefMultSyst}(b).
We thus find a group isomorphism
$Ae_dAe_j\simeq Ae_d$, for all~$1\le j\le d$, which is by construction a morphism of left~$A$-modules. This proves part (i).

Part (ii) is by definition, since $e_dAe_d=X_{dd}= I_{dd}/I_{d,d+1}$ and~$I_{dd}=R$.

For part (iii), set~$A':=A/(Ae_d A)$. It follows from the proof of part (i) that for $1\le i,j\le d-1$
$$e_iA'e_j\;=\; (I_{ij}/I_{i,d+1})/(I_{id}/I_{i,d+1})\;\simeq\; I_{ij}/I_{id}\;\simeq\; e_iA(R,\II')e_j.$$
That this yields a ring isomorphism follows again by construction.
\end{proof}

\begin{rem}
(1) Proposition~\ref{Prop4Iter}(i) implies that the left $A$-module $Ae_dA$ is projective. It seems that the corresponding claim for $Ae_dA$ as a right $A$-module fails outside of the more restrictive setting in Theorem~\ref{ThmMain}. 

(2) Even when $\II=\II^{\scJ}$ and hence $\scA(R,\II)$ is an ADR ring, the ring $\scA(R,\II')$ will not be. This is the reason why extending the class of ADR rings to the class in Section~\ref{SecDefAlg} actually makes the proof of their quasi-heredity easier.
\end{rem}

\begin{proof}[Proof of Theorem~\ref{ThmChain}]
We prove this by induction on $d$. For $d=1$, there is nothing to prove. Assume the statement is true for $d< d_0$ and consider a $d_0$-system~$\II$. By Proposition~\ref{Prop4Iter}(iii), the claim for $i<d_0$ reduces to the corresponding claims for the algebra~$\scA(R,\II')$ for the $d-1$-system~$\II'$. The statement for $i=d_0$ is precisely Proposition~\ref{Prop4Iter}(i) and (ii). \end{proof}

\subsection{The $\Delta$-modules}
Consider $A:=\scA(R,\II )$ as in Section~\ref{SecDefAlg}. 

\subsubsection{}\label{DefDelta}
For $1\le k\le d$, we define the left $A$-module
$$\Delta_k\;:=\; (A/J_{k})e_k\;=\; Ae_k/Af_{k}Ae_k.$$

\begin{lemma}\label{sesDelta}
We have a short exact sequence of~$A$-modules
$$0\to A e_{k+1}\stackrel{\psi}{\to} Ae_k\to \Delta_k\to 0,\qquad\mbox{for $1\le k\le d$,} $$
where $\psi: Ae_{k+1}\to Ae_k$ restricts to the canonical inclusion of~$I_{i,k+1}/I_{i,d+1}=e_iAe_{k+1}$ into $I_{i,k}/I_{i,d+1}=e_iAe_k$, see Lemma~\ref{LemInc}(i), for all $1\le i\le d$.
\end{lemma}
\begin{proof}
By definition of~$\Delta_k$, we have a short exact sequence
$$0\to Af_{k}Ae_k\to Ae_k\to \Delta_k\to 0.$$
It follows from Lemma~\ref{TechLem} below that $Af_{k}Ae_k=Ae_{k+1}Ae_k$.
We also have an $A$-linear morphism
$$A e_{k+1}\to Ae_{k+1}Ae_k,\qquad e_{k+1}\mapsto 1+I_{k+1,d+1}\in X_{k+1,k}= e_{k+1}Ae_k.$$
It follows from the definitions that this module morphism restricts to group isomorphisms of the form~$e_iAe_{k+1}\to e_i Ae_{k+1}Ae_k$, for all $1\le i\le d$, where both sides correspond to $I_{i,k+1}/I_{i,d+1}$. This shows the morphism is an isomorphism and composes with the inclusion of~$Ae_{k+1}Ae_k$ into $Ae_k$ to give $\psi$ in the lemma.
\end{proof}

\begin{lemma}\label{TechLem}
For $1\le k\le l\le d$, we have
$$Af_l Ae_k\;=Ae_{l+1}Ae_k.$$
\end{lemma}
\begin{proof}
By definition, $Af_l Ae_k$ is equal to $\cup_{j> l} Ae_jA_ek.$
It thus suffices to prove that 
$$Ae_jAe_k\subset Ae_iAe_k,\quad\mbox{if $k<i\le j$.}$$
The above follows easily from Definition~\ref{DefMultSyst}(b) and Lemma~\ref{LemInc}(i) 
\end{proof}

\section{The tilting bimodule~$\mT$}\label{SecTilt}

We fix a $d$-system~$\II$ of ideals in an arbitrary ring~$R$.

\subsection{Definition}

\subsubsection{}We define abelian groups
$$\mT_{kl}\;:=\; R/I_{k,d+2-l},\qquad\mbox{for all $1\le k,l\le d$.}$$
 These groups are $R$-bimodules for left and right multiplication. We also introduce the notation~$\mT_{i\ast}=\oplus_j\mT_{ij}$, $\mT_{\ast j}=\oplus_i\mT_{ij}$ and~$\mT=\oplus_{ij}\mT_{ij}$.

\subsubsection{}With $X_{ij}$ as introduced in \ref{DefX}, we have group homomorphisms
$$X_{ij}\times \mT_{kl}\to\mT_{il},\qquad (x+I_{i,d+1},r+I_{k,d+2-l})\mapsto \delta_{jk}\,(xr+ I_{i,d+2-l}),\quad\mbox{for $x\in I_{ij}$, $r\in R$}.$$
That the above homomorphism is well-defined follows from the inclusions
$$I_{i,d+1}\subset I_{i,d+2-l}\quad\mbox{and}\quad I_{ij}I_{j,d+2-l}\subset I_{i,d+2-l},$$
see Definition~\ref{DefMultSyst}(a) and Lemma~\ref{LemInc}(i). Clearly, these morphisms induce a ring morphism~$\scA(R,\II)\to\End(\mT_{\ast l})$, making~$\mT_{\ast l}$ into a left $\scA(R,\II)$-module, for all $1\le l\le d$.

\subsubsection{} Now we observe that, similarly, each group $\mT_{k\ast}$ is a left $\scA(R^{\op},\mathring{\II})$-module. For this we introduce group homomorphisms
$$e_i \scA(R^{\op},\mathring{\II})e_j \times \mT_{kl}\to \mT_{ki}, $$
by observing that
$$e_i \scA(R^{\op},\mathring{\II})e_j\;=\; I_{d+2-j,d+2-i}/I_{1,d+2-i}$$
and taking group homomorphisms
$$I_{d+2-j,d+2-i}/I_{1,d+2-i}\times \mT_{kl}\to \mT_{ki},\qquad(x+ I_{1,d+2-i},r+I_{k,d+2-l})\mapsto \delta_{jl}(rx+I_{k,d+2-i}).$$
Note that this yields indeed a {\em left} module, since $I_{d+2-j,d+2-i}$ is considered as an ideal in~$R^{\op}$.

\subsubsection{} The above gives~$\mT$ the structure of an~$\scA(R,\II)$-module and an~$\scA(R^{\op},\mathring{\II})$-module, where both actions clearly commute. It is again natural to represent $\mT$ in matrix form. For $d=2$, using the notation of Example~\ref{Ex2}, we have
$$\mT\;=\;  \left(\begin{array}{cc} R/I&R/K\\R/L&0 \end{array}\right).$$

\subsection{A double centraliser property}

\begin{lemma}
The group $\mT$ is faithful as an~$\scA(R,\II)$-module and as an~$\scA(R^{\op},\mathring{\II})$-module.
\end{lemma}
\begin{proof}
We prove that the submodule~$\mT_{\ast1}$ is already a faithful $\scA(R,\II)$-module. Take an arbitrary~$a\in \scA(R,\II)$ and set~$a_{ij}:=e_iae_j$. Furthermore, define $x_k\in\mT_{k1}$ as~$1+I_{k,d+1}$. Then $ax_k=0$ if and only if $a_{ik}=0$, for all $1\le i\le d$. The claim for $\scA(R^{\op},\mathring{\II})$ follows similarly.
\end{proof}

By the lemma, we can identify the algebras~$\scA(R,\II)$ and $\scA(R^{\op},\mathring{\II})$ with their images in~$\End(\mT)$.

\begin{prop}\label{PropDC}
We have
$$\End_{\scA(R,\II)}(\mT)=\scA(R^{\op},\mathring{\II})\qquad\mbox{and}\qquad \End_{\scA(R^{\op},\mathring{\II}) }(\mT)=\scA(R,\II).$$
\end{prop}
\begin{proof}
It is clear from construction that the actions of~$A:=\scA(R,\II)$ and~$B:=\scA(R^{\op},\mathring{\II})$ on~$\mT$ commute.
A general $\phi\in \End(\mT)$ is determined by its restrictions 
$$\phi_{kl}^{ij}\in\Hom(\mT_{ij},\mT_{kl}),\qquad\mbox{corresponding to}\quad \End(\mT)\simeq\bigoplus_{ijkl}\Hom(\mT_{ij},\mT_{kl}).$$ We now assume that $\phi$ commutes with $A$ and proceed in 5 steps to prove that $\phi\in B$.

(1) We claim that~$\phi_{kl}^{ij}=0$ unless~$i=k$. The action of the element
$$a:=1+I_{i,d+1}\in R/I_{i,d+1}=X_{ii}\subset A$$
acts as identity on~$\mT_{ij}$, for all $j$ and annihilates~$\mT_{lj}$ if $l\not=i$. Take an arbitrary~$v\in \mT_{ij}$. The condition~$\phi(av)=a\phi(v)$ implies the claim.

(2) We have $\phi^{ij}_{il}\in \Hom_R(\mT_{ij},\mT_{il})$, when considering the left $R$-module structure on~$\mT$. This follows by observing that $\phi$ commutes with arbitrary elements of~$X_{ii}=R/I_{i,d+1}\subset A$.

(3) We claim that, for all $n\le i$ and arbitrary~$j,l$, we have a commuting diagram of abelian group homomorphisms
$$\xymatrix{\mT_{nj}\ar@{->>}[rr]\ar[d]^{\phi^{nj}_{nl}}&&\mT_{ij}\ar[d]^{\phi^{ij}_{il}}\\
\mT_{nl}\ar@{->>}[rr]&&\mT_{il}
}$$
where the horizontal arrows correspond to the canonical epimorphisms induced by the inclusions of ideals in Lemma~\ref{LemInc}(ii). Indeed, this follows by observing that $\phi$ commutes with 
$$1+I_{i,d+1}\in R/I_{i,d+1}= X_{in}\subset A,$$
and that the action of this element on $\mT$ induces precisely these epimorphisms.

(4) We claim that
$$\phi^{1j}_{1l}\in \Hom_R(\mT_{1j},\mT_{1l}) =\Hom_R(R/I_{1,d+2-j},R/I_{1,d+2-l})$$ is given by right multiplication with an element in~$I_{d+2-j,d+2-l}$, which means we can identify~$\phi^{1j}_{1l}$ with the left action of an element of
$$e_lBe_j\;=\;I_{d+2-j,d+2-l}/I_{1,d+2-l}.$$
Indeed, consider the commuting diagram in (3) for $n=1$ and~$i=d+2-j$. Since we have $\mT_{d+2-j,j}=R/R=0$, the composition of~$\phi^{1j}_{1l}$ with 
$$\mT_{1l}\tto \mT_{d+2-j,l}=R/I_{d+2-j,d+2-l}$$ must be zero, which means~$\im\phi^{1j}_{1l}\subset I_{d+2-j,d+2-l}/I_{1,d+2-l}\subset \mT_{1l}$.

(5) By (1) and (3) we find that $\phi$ is completely determined by~$(\phi^{1k}_{1l})_{kl}$, meaning its restriction to the direct summand $\mT_{1\ast}$. By (4), this restriction corresponds to the action of an element of~$B$ on $\mT_{1\ast}$. Hence, we find indeed $\End_{A}(\mT)= B$.

The other direction of the double centraliser property is proved similarly.
\end{proof}

\subsection{The structure of the $\scA(R,\II)$-module~$\mT$}\label{StrucT}
In this section, we set~$A:=\scA(R,\II)$ and~$T_j:=\mT_{\ast j}$, for all $1\le j\le d$.

\begin{lemma}\label{flagT}
We have an isomorphism of~$A$-modules
$$T_j\;\simeq\; Ae_1/(Af_{d+1-j}Ae_1)=(A/J_{d+1-j})e_1,\quad\mbox{ for all $1\le j\le d$.}$$
\end{lemma}
\begin{proof}
By Lemma~\ref{TechLem} it suffices to prove that $T_j$ is isomorphic to $Ae_1/(Ae_{d+2-j}Ae_1)$. 
Consider the $A$-module morphism
$$\pi:\;Ae_1\to T_j,\quad e_1=1+I_{1,d+1}\,\mapsto\, 1+I_{1,d+2-j}\subset\mT_{1j}.$$
Restriction of this morphism to $e_iAe_1\to e_iT_j=\mT_{ij}$, for $1\le i\le d$, leads to the group epimorphism $R/I_{i,d+1}\tto R/I_{i,d+2-j}$.
The kernel of the latter is~$I_{i,d+2-j}/I_{i,d+1}=e_iAe_{d+2-j}Ae_1$. These two observation imply a short exact sequence
$$0\to Ae_{d+2-j}Ae_1\to Ae_1\stackrel{\pi}{\to} T_j\to0,$$ which concludes the proof.\end{proof}

Recall the left $A$-modules~$\Delta_k$ from \ref{DefDelta}.

\begin{lemma}\label{extT}
For all $1\le i,j\le d$, we have
$$\Ext^1_A(\Delta_i,T_j)\;=\;0.$$
\end{lemma}
\begin{proof}
By Lemma~\ref{sesDelta}, for any~$A$-module~$M$ we have $\Ext^1_A(\Delta_i,M)=0$ if and only if the group homomorphism
$$e_{i}M\;\to\; e_{i+1}M,\qquad v\mapsto av,\quad\mbox{with $a=1+I_{i+1,d}\in R/I_{id}=e_{i+1}Ae_i,$}$$
is surjective. For $M=T_j$, this homomorphism is precisely the canonical epimorphism $\mT_{ij}\tto \mT_{i+1,j}$ corresponding to the inclusion $I_{i,d+2-j}\subset I_{i+1,d+2-j}$. This concludes the proof.
\end{proof}

\begin{lemma}\label{LemNew}
For $1\le k\le d$, we have
$$\Delta_k\;\simeq\; J_{k-1}T_{d+1-k}.$$
\end{lemma}
\begin{proof}
Consider the $A$-module morphism
$$Ae_k\to T_{d+1-k},  \quad e_k\mapsto 1+I_{k,k+1}\in \mT_{k,d+1-k}\subset T_{d+1-k}.$$
By construction, it has as image $Ae_k T_{d+1-k}$.
It is clear that $f_kAe_k$ is in the kernel of this morphism since $\mT_{j,d+1-k}=R/R=0$ for $j>k$, by definition. The above morphism thus factors through $\Delta_k=(A/J_k)e_k$ and we obtain a morphism
$$\Delta_k\to T_{d+1-k},$$
which has as image $Ae_k T_{d+1-k}=J_{k-1} T_{d+1-k}$.
That this morphism is injective follows from considering the restrictions to group homomorphisms between $e_i\Delta_k\simeq I_{ik}/I_{i,k+1}$ and $\mT_{i,d+1-k}=R/I_{i,k+1}$, for all $1\le i\le d$.
\end{proof}


\section{Main results}\label{SecAlg}
From now on, we will work over the field $\mk$. It is clear that, for a $\mk$-algebra~$R$, the ring $\scA(R,\II)$ of Section~\ref{SecDefAlg} is again a $\mk$-algebra.

\subsection{Statement of results}
\subsubsection{}\label{DefLambda} Consider a {\em semisimple} $d$-system~$\II$ of ideals in a {\em $\mk$-algebra} $R$, as in Definition~\ref{DefSS}. We denote by~$\Lambda_i$ the labelling set of isoclasses of simple modules over the semisimple algebra~$R/I_{i,i+1}$, for $1\le i\le d$. Then we consider the poset
$$\Lambda:=\amalg_{i=1}^d\Lambda_i,\qquad\mbox{with}\qquad \lambda<\mu\;\; \mbox{ if and only if $\lambda\in\Lambda_i$ and~$\mu\in\Lambda_j$ with $i<j$.}$$

\begin{thm}\label{ThmMain}
For a semisimple $d$-system~$\II$ of ideals in a $\mk$-algebra~$R$, set~$A:=\scA(R,\II)$ and use notation as in \ref{DefLambda}.
\begin{enumerate}[(i)] 
\item The isoclasses of simple $A$-modules are labelled by~$\Lambda$.
\item The algebra~$(A,\le)$ is quasi-hereditary. 
\item The Ringel dual of~$(A,\le)$ is given by~$\scA(R^{\op},\mathring{\II})^{\op}$.
\item The kernel of the epimorphism $P(\lambda)\tto \Delta(\lambda)$ is projective.
\end{enumerate}
\end{thm}

\begin{rem}
In concrete situations, it is advantageous to write the elements of $\Lambda_i$ as pairs~$(i,\kappa)$, where $\kappa$ ranges over the elements of the labelling set~$\Lambda_R$ of simple $R$-modules which are not annihilated by $I_{i,i+1}\supset \scJ$. We follow this convention in Section~\ref{SecEx}.
\end{rem}

\begin{rem}
The partial order in \ref{DefLambda} shows that for an ADR algebra, the poset will be generally be far from symmetric (under $\le\;\leftrightarrow\;\le^{-1}$). Concretely, we will have $\Lambda_d=\Lambda_R$, whereas $\Lambda_1$ labels the simple $R$-modules whose projective cover has maximal Loewy length. By the definition in \ref{SecRingel}, this is a combinatorial obstruction towards having Ringel duality between ADR algebras. A module theoretic realisation of this obstruction is that $\Lambda_d$ labels projective standard modules and $\Lambda_1$ simple standard modules. An imbalance between $\Lambda_1$ and $\Lambda_d$ therefore prevents Ringel self-duality. Note that for category $\cO$ (Weyl group with Bruhat order), or the Schur algebra (the set of partitions with dominance order), the poset~$\Lambda$ is symmetric.
\end{rem}

Theorem~\ref{ThmMain} essentially remains true when we omit the condition that $\II$ be semisimple. For this we have to work in the more general realm of `stratified algebras', see \cite{CPSbook, Frisk}.
\begin{prop}\label{PropGen}
For any $d$-system~$\II$ of ideals in a $\mk$-algebra~$R$, the algebra~$\scA(R,\II)$ is (left) standardly stratified in the sense of \cite[Chapter~2]{CPSbook}. Its Ringel dual, in the sense of \cite[Section~5]{Frisk}, is given by $\scA(R^{\op},\mathring{\II})^{\op}$.
\end{prop}

\begin{rem}
Note that the convention in \cite{Frisk} would call $\scA(R^{\op},\mathring{\II})$ the Ringel dual of $\scA(R,\II)$. The advantage of that convention is that the Ringel dual is again standardly stratified. The algebra~$\scA(R^{\op},\mathring{\II})^{\op}$ has of course a {\em right} standard stratification.
\end{rem}

\subsection{Proofs of results}
\begin{proof}[Proof of Theorem~\ref{ThmMain}]
Consider the chain of idempotent ideals in Theorem~\ref{ThmChain}. By assumption, $e_i A/J_i e_i\simeq R/I_{i,i+1}$ is semisimple.
Parts (i) and (ii) then follow immediately from Lemma~\ref{LemCPS} and Theorem~\ref{ThmChain}. Note that we have $\lf(\lambda)=i$ for $\lambda\in\Lambda_i$, in the notation of~\ref{DefChain}.
 
For each $1\le i\le d$, it then follows that $Ae_i$ is a direct sum of modules $P(\nu)$ with $\nu\in\amalg_{j\ge i}\Lambda_j$, such that each module~$P(\lambda)$ for $\lambda\in \Lambda_i$ appears at least once. By Lemma~\ref{LemCPS}, the left $A$-module~$\Delta_i$, as in \ref{DefDelta}, is then a direct sum of modules isomorphic to modules in $\{\Delta(\lambda)\,|\,\lambda\in\Lambda_i\}$, such that each one appears at least once.

Now we prove part (iii). Consider the left $A$-module~$T\simeq \mT\simeq\oplus_jT_j$ of Section~\ref{SecTilt}. 
By Lemma~\ref{extT} and the previous paragraph, we have $\Ext^1_A(T,\Delta(\nu))=0$, for all $\nu\in \Lambda$. Since $A$ is quasi-hereditary, we have $Ae_1\in\cF_A^\Delta$.
By Lemmata~\ref{LemMM} and~\ref{flagT}, we then also have $T\in \cF_A^{\Delta}$, which thus means that $T$ is a direct sum of tilting modules. By Lemmata~\ref{LemMM} and~\ref{LemNew} and the above paragraph it follows that $T_{d+1-i}$ contains each module~$T(\lambda)$, for $\lambda\in\Lambda_i$ at least once as a direct summand.
Proposition~\ref{PropDC} implies that 
$$\scA(R^{\op},\mathring{\II})^{\op}\;\simeq\;\End_{A}(T)^{\op}.$$
Hence, $\scA(R^{\op},\mathring{\II})^{\op}$ is the Ringel dual of~$A$, concluding the proof of part (iii).

Part (iv) follows from the above and Lemma~\ref{sesDelta}.
\end{proof}

\begin{proof}[Proof of Proposition~\ref{PropGen}]
This is proved using the exact same arguments as in the proof of Theorem~\ref{ThmMain}. Theorem~\ref{ThmChain} shows that $\scA(R,\II)$ has a (left) standard stratification as defined in \cite[Section~2.1]{CPSbook}. The results in Section~\ref{StrucT} then show that $\mT$ is a tilting module in the sense of \cite[Section~4.2]{Frisk}, by \cite[Lemma~14]{Frisk}. Ringel duality is then precisely Proposition~\ref{PropDC}.
\end{proof}


\section{Example}\label{SecEx}

We consider the hereditary~$\mk$-algebra~$R$, defined as the path algebra of the quiver
$$\xymatrix{\bullet_\alpha\ar[rr]^{f}&&\bullet_\beta }.$$
We thus have $R=\langle \varepsilon_\alpha,\varepsilon_\beta,f\rangle$, with $f=\varepsilon_\beta f=f\varepsilon_\alpha$ and~$\scJ=\langle f\rangle$. Since the two indecomposable projective covers have different Loewy length, $1$ and $2$, this gives an example where the Ringel dual of~$\scA(R)$ was not yet described in \cite{Conde}. 
Note that we have $R\simeq R^{\op}$. 

\subsection{The ADR algebra}\label{ADRex}

We take the $2$-system given by the Jacobson system. This means
$$K=I_{12}=\langle \varepsilon_\alpha,f\rangle,\; I_{13}=0,\; L=I_{23}=\scJ, $$
and~$I_{ij}=R$ otherwise. 
In this case, $\scA(R,\II)$ is Morita equivalent to the path algebra of the quiver
$$\xymatrix{\bullet_{(2,\alpha)}\ar[rr]^{g}&&\bullet_{(1,\beta)}\ar[rr]^h&&\bullet_{(2,\beta)}&&\mbox{with relation~$h\circ g=0$.} }$$
It is then easy to obtain the (co)standard modules, leading to the conclusion
$$T(2,\alpha)\simeq P(2,\alpha),\;T(2,\beta)\simeq P(1,\beta)\;\mbox{and }\;T(1,\beta)\simeq L(1,\beta).$$
Hence $\End_{A}(T)$ is given by the path algebra of the quiver
$$\xymatrix{\bullet_{(2,\beta)'}\ar[rr]^{x}&&\bullet_{(1,\beta)'}\ar[rr]^y&&\bullet_{(2,\alpha)'}}.$$


\subsection{The Ringel dual of the Jacobson system}
Another $2$-system of ideals in~$R$ is given by
$$K=\scJ,\; I=0\;\mbox{and }\; L=\langle\varepsilon_\beta,f\rangle.$$
In this case, $\scA(R,\II)$ is the path algebra of the quiver
$$\xymatrix{\bullet_{(1,\alpha)}\ar[rr]^{x}&&\bullet_{(2,\alpha)}\ar[rr]^y&&\bullet_{(1,\beta)}}.$$
Hence, this algebra is indeed Ringel dual to the ADR algebra in Section~\ref{ADRex}, as predicted by Theorem~\ref{Main1}.

\subsection{The remaining cases}
We discuss $\scA(R,\II)$ for the remaining choices of semisimple $2$-systems $\II$ of ideals. These correspond to all triples of ideals~$K,L,I$, such that $K\supset \scJ\subset L$, $K\supset I\subset L$ and~$KL\subset I$. We ignore the 7 possibilities where $\scA(R,\II)$ is zero or semisimple. 

\subsubsection{} If we set~$K=\scJ=L$ and~$I=0$, then $\scA(R,\II)$ is isomorphic to the path algebra of the quiver
$$\xymatrix{ \bullet_{(1,\alpha)}\ar[r]^x& \bullet_{(2,\alpha)}\ar[r]^y&\bullet_{(1,\beta)}\ar[r]^z&\bullet_{(2,\beta)}&\qquad\mbox{with relation~$z\circ y=0$.}
}$$

\subsubsection{} If we set~$K=\langle \varepsilon_\alpha,f\rangle$, $L=\langle \varepsilon_\beta,f\rangle$ and~$I=0$, then $\scA(R,\II)$ is Morita equivalent to $R$.

\subsubsection{} If we set~$K=\langle \varepsilon_\alpha,f\rangle$ and~$L=\scJ=I$, then $\scA(R,\II)$ is Morita equivalent to the path algebra of the quiver
$$\xymatrix{\bullet_{(2,\alpha)}&& \bullet_{(1,\beta)}\ar[r]& \bullet_{(2,\beta)}
}$$

\subsubsection{} If we set~$K=\scJ=I$ and~$L=\langle \varepsilon_\beta,f\rangle$, then $\scA(R,\II)$ is Morita equivalent to the path algebra of the quiver
$$\xymatrix{\bullet_{(1,\beta)}&&\bullet_{(1,\alpha)}\ar[r]&\bullet_{(2,\alpha)}
}$$

\subsubsection{} If we set~$K=\scJ=L=I$, then $\scA(R,\II)$ is isomorphic to the path algebra of the quiver
$$\xymatrix{\bullet_{(1,\alpha)}\ar[r]&\bullet_{(2,\alpha)}&&\bullet_{(1,\beta)}\ar[r]&\bullet_{(2,\beta)}
}$$


\subsubsection{} If we set~$K=\langle \varepsilon_\alpha,f\rangle=L=I$, then $\scA(R,\II)$ is isomorphic to $R$.


\subsubsection{} If we set~$K=\langle \varepsilon_\beta,f\rangle=L=I$, then $\scA(R,\II)$ is isomorphic to $R$.



\subsection*{Acknowledgement}
The research was supported by ARC grant DE170100623.

\begin{flushleft}
	K. Coulembier\qquad \url{kevin.coulembier@sydney.edu.au}
	
	School of Mathematics and Statistics, University of Sydney, NSW 2006, Australia
	
	\end{flushleft}

\end{document}